\documentclass[a4 paper, 11pt]{amsart}
\usepackage[utf8]{inputenc}
\usepackage{amssymb, hyperref, enumerate, color}

\newtheorem{theorem}{Theorem}[section]
\newtheorem{proposition}[theorem]{Proposition}
\newtheorem{lemma}[theorem]{Lemma}
\newtheorem{corollary}[theorem]{Corollary}

\theoremstyle{remark}
\newtheorem{remark}[theorem]{Remark}

\numberwithin{equation}{section}

\newcommand{\R}{\mathbb{R}}
\newcommand{\C}{\mathbb{C}}
\newcommand{\xt}{(u,v,w)}

\newcommand{\lb}{\lambda}
\newcommand{\slb}{\sqrt{\lambda}}
\newcommand{\dd}{\mathrm{d}}
\newcommand{\ep}{\varepsilon}

\DeclareMathOperator{\real}{Re}
\DeclareMathOperator{\imag}{Im}
\DeclareMathOperator{\Ran}{Ran}
\DeclareMathOperator{\Ker}{Ker}
\DeclareMathOperator{\dist}{dist}
\DeclareMathOperator{\spann}{span}

\title[Energy decay in an infinite 1-D wave-heat system]{Optimal energy decay in a one-dimensional  wave-heat system with infinite heat part}

\author[A.C.S.\ Ng]{Abraham C.S.\ Ng}
\address[A.C.S.\ Ng]{St Edmund Hall, Queen's Lane, Oxford, OX1 4AR, UK}
\email{abraham.ng@maths.ox.ac.uk}

\author[D. Seifert]{David Seifert}
\address[D. Seifert]{School of Mathematics, Statistics and Physics, Herschel Building, Newcastle University, Newcastle upon Tyne, NE1 7RU, UK}
\email{david.seifert@ncl.ac.uk}

\begin{document}

\begin{abstract}
  Using recent results in the theory of $C_0$-semigroups due to Batty, Chill and Tomilov (J.\ Eur.\ Math.\ Soc.\ 18(4):853--929, 2016)  we study energy decay in a one-dimensional coupled wave-heat system with finite wave part and infinite heat part. Our main result provides a sharp estimate for the rate of energy decay of a certain class of classical solutions. The present paper can be thought of as a natural sequel to a recent work by Batty, Paunonen and Seifert (J.\ Evol.\ Equ.\ 16:649--664, 2016), which studied a similar wave-heat system with finite wave and heat parts using a celebrated result due to Borichev and Tomilov.
\end{abstract}


\maketitle

\section{Introduction}

We study the following system consisting of a wave equation on a finite interval coupled at one end with a heat equation on an infinite interval:
\begin{equation}\label{eq1}
    \begin{cases}\begin{aligned}  u_{tt}(\xi,t) &= u_{\xi\xi}(\xi,t), &&& \xi \in (-1,0),\ t>0,\\
     w_t(\xi,t) &= w_{\xi\xi}(\xi,t), &&& \xi \in (0,\infty),\  t>0, \\
     u_t(0,t) &= w(0,t), &u_\xi(0,t) &= w_\xi(0,t), \quad& t>0, \\
     u(-1,t) &= 0, & && t>0, \\
     u(\xi,0)&=u(\xi), & u_t(\xi,0) &= v(\xi),  & \xi \in (-1,0), \\
     w(\xi,0) &= w(\xi), &&& \xi \in (0,\infty).
    \end{aligned}\end{cases}
\end{equation}
Here we assume that the initial data $u,v, w$ satisfy $u \in H^1(-1,0), v \in L^2(-1,0)$ and $w \in L^2(0,\infty)$. Given initial data $x=(u,v,w)$, we define the energy of the solution corresponding to $x$ as
\begin{equation}
\label{eq:energy}
E_x(t) = \frac{1}{2}\int_{-1}^\infty |u_\xi(\xi,t)|^2 + |u_t(\xi,t)|^2 + |w(\xi,t)|^2 \, \dd\xi, \qquad t\geq 0.
\end{equation}
All functions on the right-hand side have been extended by zero in $\xi$ to the interval $(-1,\infty)$. Assuming sufficient regularity of the solution, a simple calculation via integration by parts shows that
$$E'_x(t) =  - \int_0^\infty|w_\xi(\xi,t)|^2\, \dd\xi, \qquad t\geq0,$$  
so the energy is non-increasing in time. The main contribution of this paper is to obtain a sharp estimate on the \emph{rate} of this energy for a natural class of sufficiently well-behaved solutions.

A finite interval version of the above problem (with a different coupling condition) was studied in \cite{ZZ1}, where both the wave and the heat components were of unit length and had Dirichlet boundary conditions, yielding the optimal decay rate $E_x(t)= O(t^{-4})$. The method used in \cite{ZZ1} relied on a rather intricate spectral analysis, which was required in order to apply the theory of Riesz spectral operators. The same finite interval wave-heat system, this time with Neumann boundary for the wave part (and with the same coupling condition as in \eqref{eq1}), was studied in \cite{BPS1}. In contrast to \cite{ZZ1}, however, the approach in \cite{BPS1} was based on the methods of non-uniform stability \cite{BaDu, BoTo}, which greatly simplified the analysis necessary to obtain this sharp rate of decay.  The abstract methods in \cite{BaDu, BoTo} depend on the imaginary axis lying within the resolvent set of the generator of the solution semigroup. This was generalised in \cite{BaChTo, ChSe, Mar11, RSS} to allow for the case where the spectrum of the generator touches the imaginary axis at zero. A similar approach has recently been used in \cite{BPS2} to study a  wave-heat system on a rectangular domain. For surveys of similar problems we refer the interested reader to \cite{AvTr, BPS1}.

In our paper, we follow the approach of \cite{BPS1}, but where the authors of \cite{BPS1} appeal to the main result of \cite{BoTo} we use instead the generalised result due to Batty, Chill and Tomilov \cite{BaChTo}, finding in the main result of the paper, Theorem \ref{opt}, a class of classical solutions to \eqref{eq1} which satisfy
$$E_x(t) = O(t^{-2}), \qquad t \to \infty.$$ We further establish that this rate sharp. This shows that extending the heat part of the coupled wave-heat system of \cite{BPS1} to infinity slows the rate of energy decay by a factor of $t^{2}$. The crucial difference between the infinite and the finite systems is that the damping provided by the heat part is significantly weaker in the infinite case.

The paper is organised as follows. First, in Section~\ref{sec:gen}, we show that our problem can be recast as an abstract Cauchy problem and solved in the sense of $C_0$-semigroups. We also provide a detailed description of the spectrum of the semigroup generator. In Section~\ref{sec:res}, we establish sharp upper bounds for the norm of the resolvent operator along the imaginary axis, both near zero and at infinity. Then, in Section~\ref{sec:en}, we apply the Batty-Chill-Tomilov result to deduce an optimal estimate for the rate of energy decay for a certain class of classical solutions -- namely, for solutions with initial data lying in both the domain and the range of $A$. We moreover provide a  characterisation of the range of $A$. Finally, in Section~\ref{sec:Neumann}, we consider the case where the wave part of the coupled system satisfies a Neumann boundary condition. We show that the approach taken in the Dirichlet case now leads to an unbounded semigroup, and we consider an alternative formulation of the problem which allows us to recover our main results in the Neumann boundary case.

We use standard notation, closely following that used in \cite{BPS1}.  We thus denote by $D(A),$ $\Ker(A),$ $\Ran(A)$, $\sigma(A)$, and $\rho(A)$ the domain, kernel, range, spectrum and resolvent set, respectively, of a closed linear operator $A$ acting on a Hilbert space (assumed always to be complex). The resolvent operator $(\lb-A)^{-1}$, for $\lb \in \rho(A)$, will usually be written as $R(\lb,A)$. Given $\lb \in \C$, we define the square root $\sqrt{\lb}$ by taking the branch cut along the negative real axis, that is, for $\lb = re^{i\theta}$ where $r\ge0$ and $\theta \in (-\pi,\pi]$, we let $\sqrt{\lb} = r^{1/2}e^{i\theta/2}$. We denote the open complex left half-plane by $\C_-$. Given functions $f,g\colon (0,\infty)\to\R_+$ and $a\in [0,\infty]$, we write $f(t) = O(g(t)),\ t\to a,$ to indicate that $f(t)\leq Cg(t)$ for some constant $C>0$ and all $t>0$ sufficiently close to $a$ (or sufficiently large in the case $a = \infty$). We  write $f(t) \asymp g(t), $ $t\to a$, if both $f(t) = O(g(t))$ and $g(t) = O(f(t))$ as $t\to a$. If $g(t) >0$ for all sufficiently large $t>0$, we write $f(t) = o(g(t)),\ t\to\infty$, if $f(t)/g(t)\to0$ as $t\to \infty$.  We treat the case for functions defined on $\R$ or $\R\setminus\{0\}$ analogously. Finally, if $p$ and $q$ are real-valued quantities we will often use the notation $p \lesssim q$ to mean that $p \leq Cq$ for some constant $C >0$.

\subsection*{Acknowledgements}   Both authors thank Hansen Chen and Charles Batty for helpful discussions on the topic of this paper. The first author is grateful  to the University of Sydney for funding this work through the Barker Graduate Scholarship.

\section{Well-posedness -- the semigroup and its generator}\label{sec:gen}

In this section, we first prove that \eqref{eq1} is well-posed with solutions given by the orbits of a $C_0$-semigroup of contractions, before turning to look at the spectrum of the semigroup generator.

\subsection{Existence of the semigroup}

Our first step is to recast \eqref{eq1} as an abstract Cauchy problem. Consider the Hilbert space 
$$
X=\big\{(u,v,w) \in H^1(-1,0)\times L^2(-1,0)\times L^2(0,\infty): u(-1) = 0\big\},$$
endowed with the norm $\|\cdot\|_X$ given by
$$\|x\|_X^2 = \|u'\|_{L^2(-1,0)}^2 + \|v\|_{L^2(-1,0)}^2+ \|w\|_{L^2(0,\infty)}^2$$
for $x=(u,v,w)\in X$, and let
$$X_0=\big\{(u,v,w) \in H^2(-1,0)\times H^1(-1,0) \times H^2(0,\infty) : u(-1) =0\big\}.$$
We define the operator $A$ on $X$ by $Ax = (v,u'',w'')$ for $x=(u,v,w)$ in the domain
\begin{equation*}D(A) = \big\{(u,v,w) \in X_0 :  v(-1) = 0,\ u'(0) = w'(0),\ v(0) = w(0)\big\}.\end{equation*}

\begin{lemma}\label{biglemma1} The following hold:
\begin{enumerate}
    \item[\textup{(a)}] $A$ is closed;
    \item[\textup{(b)}] $A$ is densely defined;
    \item[\textup{(c)}] $A$ is dissipative;
    \item[\textup{(d)}] $I-A$ is surjective.
\end{enumerate}
\end{lemma}

\begin{proof}
(a) Let $x_n = (u_n,v_n,w_n) \in D(A)$, $n\ge1$, be such that $$x_n\to x = (u,v,w),\qquad  Ax_n = (v_n,u_n'',w_n'') \to y = (f,g,h)$$ in $X$ as $n\to\infty$. Then $u_n$ converges to $u$ in $H^1(-1,0)$ and $u_n''$ converges to $g$ in $L^2(-1,0)$ as $n\to\infty$. Hence
\begin{equation}\label{wkd}\int_{-1}^0 u\varphi'' = \lim_{n\to \infty}\int_{-1}^0 u_n\varphi'' = \lim_{n\to\infty}\int_{-1}^0 u_n'' \varphi = \int_{-1}^0 g\varphi \end{equation} for all $\varphi \in C_c^\infty(-1,0)$,  so that $u \in H^2(-1,0)$ and $u'' = g$. As $v_n$ converges to both $v$ and $f$ in $L^2(-1,0)$ as $n\to\infty$, we have $v=f$. In particular, $v \in H^1(-1,0)$. Next, $w_n$ converges to $w$ and $w_n''$ to $h$ in $L^2(0,\infty)$ as $n\to\infty$. Standard Sobolev theory (see \cite[page ~217]{Brezis}) ensures the existence of a constant $C>0$ such that $$\|\psi'\|_{L^2(0,\infty)} \leq \|\psi''\|_{L^2(0,\infty)} + C\|\psi\|_{L^2(0,\infty)}$$ for all $\psi \in H^2(0,\infty)$, so that $(w_n')$ is Cauchy and converges to some $h_0$ in $L^2(0,\infty)$. Using similar reasoning to that in \eqref{wkd}, we see that $w\in H^2(0,\infty)$ with $w' = h_0$ and $w'' = h$. By passing to a subsequence for which all the components converge almost everywhere we may verify that $x$ satisfies the necessary coupling conditions to be in the domain $D(A)$. It follows that $Ax = y$, and hence $A$ is closed.\smallskip

\noindent (b) The linear functional $\phi_0 \colon x = (u,v,w) \mapsto v(-1)$ is unbounded on $X_0$, and hence
$$X_1 = \Ker \phi_0 = \big\{(u,v,w) \in X_0 : v(-1)=0 \big\}$$ is dense in $X_0$. Similarly,
$X_2 = \Ker \phi_1 $ is dense in $X_1$, where $\phi_1$ is the unbounded linear functional on $X_1$ defined by $x \mapsto v(0)-w(0)$, and by considering the unbounded linear functional $\phi_2 \colon x \mapsto u'(0) - w'(0)$ on $X_2$, we see that $X_3 = \Ker \phi_2 $ is dense in $X_2$. Thus we have a decreasing finite chain of subspaces $$X \supset X_0 \supset X_1 \supset X_2 \supset X_3 = D(A),$$ where each subspace is dense in the preceding one under the norm of $X$. Hence $A$ is densely defined.\smallskip

\noindent (c) Let $x=(u,v,w) \in D(A)$. A straightforward calculation yields 
$$\langle Ax,x\rangle_X  = v(0)\overline{u'(0)} -w'(0)\overline{w(0)} - \int_{-1}^0v\overline{u''} + \int_{-1}^0 u''\overline{v} -\int_0^\infty w'\overline{w'},$$ and hence $$ \real{\langle Ax,x\rangle}_X = -\|w'\|_{L^2}^2 \leq 0,$$ so $A$ is dissipative. Here and in what follows we omit the intervals for function spaces appearing as subscripts, as these will always be clear from the context.\smallskip

\noindent (d) We perform a procedure here with general $\lambda \in\C\setminus(-\infty,0]$ which will reappear in later sections. For the purposes of this section it would be sufficient to consider $\lambda = 1$. Given $y = (f,g,h) \in X$, we wish to find an $x = \xt\in D(A)$ such that $(\lambda-A)x = y$, which leads to the system \begin{subequations}\begin{align} u'' & = \lambda^2 u-\lambda f -g, & \xi \in (-1,0), \label{eqa}\\ v & = \lambda u-f, & \xi \in (-1,0), \\  w'' & = \lambda w-h, & \xi \in (0,\infty), \label{eqc} \\u(-1) &= v(-1) = 0, \ v(0)  = w(0),\ u'(0) = w'(0).
\end{align}\end{subequations}
Following the proof of \cite[Theorem~3.1]{BPS1}, the general  solution of  \eqref{eqa} subject to $u(-1) = 0$ is easily seen to be
\begin{equation}\label{equ}
    u(\xi) = a(\lambda)\sinh(\lambda(\xi+1)) - U_\lambda(\xi), \qquad\xi \in [-1,0],
\end{equation} where $a(\lambda) \in \C$ is a constant free to be determined later and 
$$U_\lambda(\xi) =  \frac{1}{\lambda}\int_{-1}^\xi \sinh(\lambda(\xi-r))(\lambda f(r)+g(r))\, \dd r, \qquad \xi \in [-1,0].$$
We thus have
\begin{equation}\label{equ'}
    u'(\xi) = \lambda a(\lambda)\cosh(\lambda(\xi+1)) - U_\lambda'(\xi), \qquad \xi \in [-1,0].
\end{equation}  
Clearly, $u\in H^2(-1,0)$ and hence $v\in H^1(-1,0)$ with $v(-1)=\lambda u(-1)-f(-1) =0$. To find the general solution for \eqref{eqc}, note first that by taking the  branch cut along  $(-\infty,0]$ we ensure that ${\real{\sqrt{\lb}}} >0$ for all $\lb \in \C\setminus(-\infty,0]$, and hence we may define the Green's function $G_\lambda\in L^1(\R)$ by
$$G_\lambda(\xi) = \frac{1}{2\sqrt{\lb}}e^{-\sqrt{\lb}|\xi|},\qquad \xi\in\R.$$ We now define the function $w$ by
\begin{equation}\label{eqw}w(\xi) = (G_\lambda*h)(\xi) + b(\lb)e^{-\slb \xi}, \qquad\xi \in [0,\infty),\end{equation} 
where $b(\lambda)\in\C$ is another free parameter to be chosen shortly.  Young's inequality for convolutions then implies that $w \in H^2(0,\infty)$ and, letting
$$W_\lambda(\xi)=\frac{1}{2\slb}e^{\slb \xi} \int_\xi^\infty h(r)e^{-\slb r}\,\dd r,\qquad \xi\in[0,\infty),$$
we have
$$(G_\lambda*h)(\xi)=W_\lambda(\xi)+\frac{1}{2\slb}e^{-\slb \xi} \int_0^\xi h(r)e^{\slb r}\,\dd r,\qquad\xi \in [0,\infty).$$
A simple calculation now shows that $w$ solves \eqref{eqc}. 

It remains to determine the constants $a(\lambda)$ and $b(\lambda)$. Using \eqref{equ} and \eqref{eqw}, the coupling condition $\lambda u(0)-f(0) = v(0) = w(0)$ gives
$$\lambda a(\lambda)\sinh(\lambda) - b(\lambda) = \lambda U_\lambda(0) + f(0) + W_\lb(0).$$ Likewise, the condition $u'(0) = w'(0)$ is equivalent to $$\lambda a(\lambda)\cosh(\lambda)+b(\lambda)\sqrt{\lambda} = U_\lambda'(0) + \slb W_\lb(0).$$ We write these two equations in matrix form as
\begin{equation}\label{eqM}
\begin{pmatrix}
    \lambda\sinh(\lambda) & -1\\ \lambda \cosh(\lambda) & \slb
    \end{pmatrix}  \begin{pmatrix} a(\lambda) \\ b(\lambda)
     \end{pmatrix} = \begin{pmatrix} \lambda U_\lambda(0) + f(0) + W_\lb(0) \\ U_\lambda'(0) + \slb W_\lb(0)
     \end{pmatrix},
\end{equation}
Hence, for $\lb \in \C\setminus(-\infty,0]$, $\lb - A$ is surjective if and only if \eqref{eqM} has a solution for any given $y = (f,g,h)\in X$, which in turn is equivalent to $\det M_\lb\ne0$, where $M_\lb$ is the matrix appearing on the left-hand side of \eqref{eqM}. Note that 
$$\det M_\lb = \lb\big(\cosh(\lb)+\slb \sinh(\lb)\big).$$  
Since $\det M_1 =e \neq 0,$
    we see that $I-A$ is surjective, as required.
\end{proof}

The following result is an immediate consequence of Lemma~\ref{biglemma1} and the  Lumer-Phillips theorem.

\begin{theorem}\label{wpthm}
$A$ generates a contractive $C_0$-semigroup $(T(t))_{t\geq0}$ on $X$.
\end{theorem}

\subsection{Spectrum of the generator}

From Theorem \ref{wpthm}, we know that $\sigma(A)$ is contained in the closed left half-plane. However, we can say more about the spectrum.

\begin{theorem}\label{spectrum}
The spectrum of $A$ is given by the disjoint union 
$$\sigma(A) = (-\infty,0]\cup\sigma_p(A),$$ 
where the point spectrum is given by 
\begin{align*}\sigma_p(A) = \big\{\lb \in \C_- : \cosh(\lb)+\slb\sinh(\lb)  = 0\big\}.\end{align*}
In particular, the spectrum satisfies $\sigma(A) \cap i\R = \{0\}$.
\end{theorem}

\begin{proof}
We first show that $(-\infty,0]$ is in the spectrum but contains no eigenvalues. Let $\lb \in (-\infty,0]$ and define
$$h_n^\lb(\xi) = \frac{e^{i\sqrt{-\lb}\xi}}{\sqrt{n}}\Phi\left(\frac{\xi}{n}-1\right), \qquad \xi \geq 0,\ n\ge1,$$ where $\Phi\colon\R\to[0,1]$ is a smooth bump function such that $\Phi (\xi)= 1$ for $|\xi|\le1/2$ and $\Phi (\xi)= 0$ for $|\xi|\ge1$. Then $x_n = (0,0,h_n^\lb) \in D(A)$ and, moreover,
$$\|x_n\|_X^2 = \frac{1}{n}\int_0^\infty \left|\Phi\left(\frac{\xi}{n} -1\right)\right|^2 \dd \xi \geq \frac{1}{n}\int_{n/2}^{3n/2} \dd\xi = 1,\qquad n\ge1.$$ For $\xi\ge0$ and $n\ge1$ we have $h_n^\lb(\xi) = e^{i\sqrt{-\lb}\xi}h_n^0(\xi)$ and
$$\|(h_n^0)'\|_{L^2}^2 \leq  \frac{2}{n^2}\|\Phi'\|_\infty^2, \qquad \|(h_n^0)''\|_{L^2}^2 \leq \frac{2}{n^4}\|\Phi''\|_\infty^2.$$ Hence
$$  \|(\lb-A)x_n\|_X  \leq 2 \sqrt{-\lb}\,\|(h_n^0)'\|_{L^2} +\|(h_n^0)''\|_{L^2} \to 0,\qquad n\to\infty,$$ 
so $\lb$ is an approximate eigenvalue of $A$. If $(\lb-A)x = 0$ for some  $x=(u,v,w) \in D(A)$, then $u''=\lambda^2u$, $v=\lambda u$ and $w''=\lambda w$. Since $\lambda\le 0$, the only solution of the third equation which lies in $L^2(0,\infty)$ is $w=0$, and it then follows easily that $u=0$, and therefore $x=0$. This proves that $(-\infty,0]$ lies in the spectrum of $A$ but contains no eigenvalues.

Assume now that $\lb \in {\overline{\C_-}}\setminus\{0\}$. Following the proof of Lemma~\ref{biglemma1}(d) we see that $(\lb - A)x = 
0$ has a non-zero solution $x\in D(A)$ if and only if $\det M_\lambda=0$, which is equivalent to $ \cosh(\lb)+{\slb\sinh(\lb)}  =0$. The proof of Lemma \ref{biglemma1}(d) further shows that $\lb-A$ is surjective whenever $\cosh(\lb)+{\slb\sinh(\lb)} \neq 0$. Thus $\lb-A$ is invertible whenever $  \cosh(\lb)+{\slb\sinh(\lb)}\neq 0$. To prove that this does not occur for non-zero $\lb$ on the imaginary axis, let $F(s)=\cos(s)+i{\sqrt{is}\sin(s)}$ for $s \in \R\setminus\{0\}$. Then  $\real{F(s)}\ne0$ whenever $\sin(s)=0$, while  $\sin(s)\ne0$ implies that $\imag{F(s)} \neq 0$. Thus $F(s)\ne0$ for all $s\in\R\setminus\{0\}$, and it follows that $\sigma(A)\cap i\R = \{0\}$.
\end{proof}

Note that this result would permit us to obtain explicit asymptotic expansions for the eigenvalues of $A$. We do not pursue this here, however, and indeed it is a strength of our method that no detailed information about the eigenvalues or the corresponding eigenvectors is required; see however Remark~\ref{rouche}(b) below.

\section{Resolvent bounds}\label{sec:res}

Here we obtain sharp upper bounds on the growth of $\|R(is,A)\|$ as $|s|$ tends to zero and infinity. The main result is Theorem~\ref{mainthm} below, which in Section~\ref{sec:en} will be crucial in deriving an optimal estimate for the rate of energy decay of sufficiently well-behaved solutions to \eqref{eq1}.

\begin{theorem}\label{mainthm}The following hold:
\begin{enumerate}
    \item[\textup{(a)}] $\|R(is,A)\| = O(|s|^{1/2})$ as $|s| \to \infty$;
    \item[\textup{(b)}]  $\|R(is,A)\| = O(|s|^{-1})$ as $|s| \to 0$.
\end{enumerate}
\end{theorem}

We begin with two simple technical lemmas; see \cite[Lemma 3.3]{BPS1} for a proof of the first.

\begin{lemma}\label{1U}
There exists a constant $C\geq 0$ such that, for all $f \in H^1(-1,0),$ $g\in L^2(-1,0)$ and $s\in \R$,
\begin{align*}
\bigg|\int_{-1}^\xi \sin(s(\xi-r))(isf(r)+g(r))\,\dd r\bigg| & \leq C\|f\|_{H^1}+\|g\|_{L^2}, \qquad \xi \in [-1,0], \\
\bigg|\int_{-1}^\xi \cos(s(\xi-r))(is f(r)+g(r))\,\dd r\bigg| & \leq C\|f\|_{H^1}+\|g\|_{L^2}, \qquad \xi \in [-1,0].
\end{align*}
\end{lemma}

\begin{lemma}\label{lem1}The following hold:
\begin{enumerate}
  \item[\textup{(a)}]  $|\cos(s)+i\sqrt{is}\sin(s)| \geq \frac{1}{3}$ for $|s|$ sufficiently large;
  \item[\textup{(b)}] $|\cos(s)+i\sqrt{is}\sin(s)| \geq \frac{1}{3}$ for $|s|$ sufficiently small.
\end{enumerate}
\end{lemma}

\begin{proof}
(a) Note that 
$$\big|\imag\big(\cos(s)+i\sqrt{is}\sin(s)\big)\big|=\sqrt{\frac{|s|}{2}}|\sin(s)|.$$
Hence if $\sqrt{|s|/2} \,|\sin(s)|\ge1/3$ then the required inequality holds. On the other hand, if $\sqrt{|s|/2} \,|\sin(s)|<1/3$ and if $|s|\ge1$, then 
$$\big|\cos(s)+i\sqrt{is}\sin(s)\big|\ge |\cos(s)|-\sqrt{|s|}\,|\sin(s)|>\frac{\sqrt{7}-\sqrt{2}}{3}>\frac{1}{3},$$
which proves part (a).
\smallskip

\noindent (b) This is clear since $\cos(s) \to 1$ and $i\sqrt{is}\sin(s) \to 0$  as $|s|\to 0$.
\end{proof}

\begin{proof}[Proof of Theorem \ref{mainthm}]
Let $s\in \R\setminus\{0\}$ and $y= (f,g,h) \in X$, further defining $x = (u,v,w) \in D(A)$ by $x = R(is,A)y.$ Since $v = is u -f$, we see that
$$ \|x\|_X \lesssim \|s u\|_{L^2} + \|u'\|_{L^2} +\|w\|_{L^2}+\|f\|_{L^2}.$$ 
Here and in the remainder of the proof the implicit constant is independent of both $s$ and $y$. The result will follow once we have established estimates for each of the first three summands on the right-hand side of the above equation. Throughout this proof we use the notation introduced in the proof of Lemma~\ref{biglemma1}(d), although for brevity we write $a_s$ instead of $a(is)$ and $b_s$ instead of $b(is)$. 

Consider $u$ given by \eqref{equ} for $\lambda=is$ with $s\in \R\setminus\{0\}$. By Lemma~\ref{1U}, it is enough to consider $|s a_s|$ in order to estimate $\|su\|_{L^2}$ and $\|u'\|_{L^2}$. Inverting the matrix $M_{is}$ in \eqref{eqM} we obtain
$$ \begin{pmatrix}a_s \\ b_s\end{pmatrix}=   \frac{1}{\det M_{is}}\begin{pmatrix}\sqrt{is} & 1 \\ -is\cos(s) & -is\sin(s)\end{pmatrix}\begin{pmatrix}is U_{is}(0) + f(0) + W_{is}(0) \\ U_{is}'(0) + \sqrt{is} W_{is}(0)\end{pmatrix}.$$
Hence by Lemmas \ref{1U} and \ref{lem1} we have
$$\begin{aligned}
    |sa_s| & \lesssim\big|\sqrt{is}\big(isU_{is}(0)+f(0) +W_{is}(0)\big)+ U_{is}'(0) + \sqrt{is}W_{is}(0)\big|  \\
    & \lesssim (1+|s|^{1/2})(\|f\|_{H^1} + \|g\|_{L^2}) + |s|^{1/2}|W_{is}(0)| \label{sas},
\end{aligned}$$
for $s \in \R\setminus\{0\}$. Noticing that $   |W_{is}(0)|  \lesssim |s|^{-3/4}\|h\|_{L^2}$
this gives 
$$|sa_s| \lesssim (1+|s|^{1/2}+|s|^{-1/4})\|y\|_X,\qquad s \in \R\setminus\{0\}.$$
We now estimate $w$, which is given by \eqref{eqw} for $\lambda=is$ with $s\in \R\setminus\{0\}$. Since
\begin{equation}\label{west}\|w\|_{L^2} \leq  \|G_{is}\|_{L^1}\|h\|_{L^2}+\frac{|b_s|}{(2|s|)^{1/4}},\end{equation}
 it suffices to estimate $\|G_{is}\|_{L^1}$ and $|b_s|$ for $\in \R\setminus\{0\}$. A simple calculation gives  $ \|G_{is}\|_{L^1} \lesssim |s|^{-1}$, and estimating the expression obtained for $b_s$ from the above matrix equation with the aid of Lemmas \ref{1U} and \ref{lem1} yields
$$\begin{aligned}    |b_s|  &\lesssim \|f\|_{H^1}+\|g\|_{L^2} + (1+|s|^{1/2})|W_{is}(0)| \\&\lesssim \|f\|_{H^1}+\|g\|_{L^2} +(|s|^{-3/4}+|s|^{-1/4})\|h\|_{L^2}\end{aligned}$$
for $s\in \R\setminus\{0\}$. Putting together the pieces we obtain $\|x\|_X\lesssim(1+|s|^{1/2}+|s|^{-1})\|y\|_X$ and hence $\|R(is,A)\|\lesssim 1+|s|^{1/2}+|s|^{-1}$ for $s\in \R\setminus\{0\}$. The results now follow.
\end{proof}

\begin{remark}\label{rouche} 
(a) In fact, we have $\|R(is,A)\| \asymp |s|^{-1}$ as $|s| \to 0$. This follows from the upper bound just proved together with Theorem~\ref{spectrum} and  the elementary inequality $$\|R(\lb,A)\| \geq \frac{1}{\dist(\lb,\sigma(A))}, \qquad  \lb \in \rho(A).$$ \smallskip

\noindent (b) Furthermore, it can be shown that
\begin{equation*}\label{eq:rouche}
\limsup_{|s|\to\infty}|s|^{-1/2}\|R(is,A)\| >0.
\end{equation*}
This is done by looking at the distribution of eigenvalues of $A$ in the left half plane and observing by means of  Rouch\'{e}'s theorem at what rate they approach the imaginary axis; see \cite[Theorem 3.4]{BPS1} for a proof which can  easily be adapted to our case. It follows that for any positive function $r$ such that $r(s)=o(|s|^{1/2})$ we have $\|R(is,A)\| \ne O (r(|s|))$ as $|s|\to\infty$. 
\end{remark}

\section{Quantified energy decay}\label{sec:en}

We now convert the resolvent estimates of Theorem~\ref{mainthm} into a rate of energy decay of a certain class of classical solutions which satisfy \eqref{eq1}. The following theorem is the main result of our paper, and its proof relies crucially on recent abstract results obtained in \cite{BaChTo}.

\begin{theorem}\label{opt}
For all $x\in X$ we have $E_x(t) \to0$ as $t\to\infty$, and if $x \in D(A)\cap \Ran(A)$ then $E_x(t) = O(t^{-2})$ as $t\to \infty$. Moreover, this rate is optimal in the sense that, given any positive function $r$ satisfying $r(t) = o(t^{-2})$ as $t \to \infty$, there exists $x \in D(A)\cap\Ran(A)$ such that $E_x(t) \neq O(r(t))$ as $t \to \infty$.
\end{theorem}

\begin{proof}
The first part follows from Theorem~\ref{spectrum} and the well-known countable spectrum theorem  for bounded $C_0$-semigroups on reflexive spaces due to Arendt, Batty, Lyubich and V\~u; see for instance \cite[Corollary~2.6]{AreBat88}.  Meanwhile by \cite[Theorem 8.4]{BaChTo} we have 
$$\|T(t)A(I-A)^{-2}\| = O(t^{-1}),\qquad t\to\infty.$$ Note also that $\Ran(A(I-A)^{-2})= D(A)\cap \Ran(A)$. Hence for any $x \in D(A)\cap \Ran(A)$, there exists $y \in X$ such that $x = A(I-A)^{-2}y$, and hence
$$E_x(t) = \frac{1}{2} \|T(t)x\|_X^2 =O(t^{-2}),\qquad t\to\infty.$$ 

To prove optimality, suppose there exists a positive function $r$ such that $r(t)=o(t^{-2})$ as $t\to \infty$ and that for all $x\in D(A)\cap \Ran(A)$ we have $E_x(t) = O(r(t))$ as $t\to \infty$. Then a simple application of the uniform boundedness principle shows that
$$\|T(t)A(I-A)^{-2}\|=O(r(t)^{1/2}),\qquad t\to\infty,$$ 
and hence
$$\limsup_{t\to \infty}t\|T(t)A(I-A)^{-2}\| =0.$$
The analogue of \cite[Theorem~6.9]{BaChTo} discussed on \cite[page~923]{BaChTo} implies that $0\in\sigma_p(A)$, which contradicts Theorem~\ref{spectrum}.
\end{proof}

\begin{remark}\label{rem:opt}
(a) In fact, \cite[Theorem 8.4]{BaChTo} also yields the finer statement 
$$\|T(t)A(I-A)^{-3/2}\| = O(t^{-1}), \qquad t\to \infty.$$ Furthermore,  \cite[Proposition 3.10]{BaChTo} shows that $$\Ran(A(I-A)^{-3/2}) = D(A^{1/2})\cap\Ran(A).$$ Using the semigroup property, it is easily seen that $$\|T(t)x\| = O(t^{-2k}), \qquad t\to \infty,$$ for all $x \in D(A^k)\cap\Ran(A^{2k})$ and all integers $k\geq1$. In other words, smoother orbits decay faster in a way that depends on the resolvent estimates at zero and at infinity. Thus, when it comes to initial data with higher regularity, the optimality of the resolvent estimate in Theorem~\ref{mainthm}(a) as $|s|\to \infty$ plays an important role that is not seen for general initial data in $ D(A)\cap \Ran(A)$, where the resolvent estimate for $\|R(is,A)\|$ as $|s| \to 0$ is the sole determining factor for the rate of energy decay.\smallskip

\noindent (b) Note that the rate of energy decay obtained in Theorem~\ref{opt} is slower by a factor of $t^2$ than the rate obtained for the wave-heat system with a finite heat part, both when the wave equation satisfies a Dirichlet boundary condition at $\xi=-1$ and when it satisfies a Neumann boundary condition at $\xi=-1$; see \cite{BPS1, ZZ1}.  Thus the finite heat equation provides a stronger damping than the infinite heat equation. This is unsurprising, of course, given that the uncoupled heat on a  bounded interval equation gives rise to an exponentially stable semigroup, whereas the semigroup corresponding to the heat equation on the half-line is a bounded analytic semigroup but not exponentially stable.
\end{remark}

We finish this section by providing a characterisation for the range $\Ran(A)$ of $A$ in order to better understand the space of classical solutions $D(A)\cap\Ran(A)$ of initial data for which we have the sharp decay estimate.

\begin{proposition}\label{ranA}
Let $(f,g,h) \in X$. Then $(f,g,h) \in \Ran(A)$ if and only if the following hold:
\begin{enumerate}
    \item[\textup{(a)}] $\xi \mapsto \displaystyle\lim_{a\to \infty}\int_\xi^a h(r)\, \dd r \in L^2(0,\infty)$;
    \item[\textup{(b)}]  $\xi \mapsto \displaystyle\lim_{b\to \infty}\int_\xi^b \left(\displaystyle\lim_{a\to \infty}\int_t^a h(r)\, \dd r\right)\dd t \in L^2(0,\infty)$;
    \item[\textup{(c)}]  $f(0) = \displaystyle\lim_{b\to \infty}\int_0^b \left(\displaystyle\lim_{a\to \infty}\int_t^a h(r)\, \dd r\right)\dd t$.
\end{enumerate}
\end{proposition}

\begin{proof}
Suppose that there exists $(u,v,w) \in D(A)$ such that $A(u,v,w) = (f,g,h)$. Then we have
$$u''=g,\ v=f,\ w''=h,\ u(-1)=0,\ u'(0)=w'(0),\ v(0)=w(0).$$
Moreover, $w$ and $w'$ are continuous and decay to zero at infinity. Hence the identities
$$w'(a) - w'(\xi) = \int_\xi^a h(r)\, \dd r, \qquad w(b)-w(\xi) = \int_\xi^b w'(t)\, \dd t$$ imply that the improper integrals of $h$ and $w'$ exist and that
$$w'(\xi) = -\lim_{a \to \infty}\int_t^a h(r)\, \dd r, \qquad  w(\xi) = \displaystyle\lim_{b\to \infty}\int_\xi^b \left(\displaystyle\lim_{a\to \infty}\int_t^a h(r)\, \dd r\right)\dd t.$$ In particular, conditions (a) and (b) are satisfied. Since $f=v$ the coupling condition $v(0) = w(0)$ implies that
$$f(0)= \displaystyle\lim_{b\to \infty}\int_0^b \left(\displaystyle\lim_{a\to \infty}\int_t^a h(r)\, \dd r\right)\dd t,$$ so (c) holds as well.

For the converse, suppose that there exists $(f,g,h)\in X$ such that properties (a), (b) and (c) hold. Let $w \in L^2(0,\infty)$ be given by $$w(\xi) = \displaystyle\lim_{b\to \infty}\int_\xi^b \left(\displaystyle\lim_{a\to \infty}\int_t^a h(r)\, \dd r\right)\dd t, \qquad \xi >0.$$ Then
$$w'(\xi) = -\displaystyle\lim_{a\to \infty}\int_\xi^a h(r)\, \dd r,\qquad \xi>0, $$ and 
$w'' = h$, so  $w \in H^2(0,\infty)$. Defining $v = f \in H^1(-1,0)$, we have $v(-1) = f(-1) = 0$ and $$v(0) = f(0) = \lim_{b\to \infty}\int_0^b \left(\lim_{a\to \infty}\int_t^a h(r)\,\dd r\right)\dd t = w(0).$$ Finally, let $u \in L^2(-1,0)$ be given by
$$u(\xi) = \int_{-1}^{\xi}\int_{-1}^t g(r)\, \dd r \,\dd t + \left(w'(0) - \int_{-1}^0 g(r)\, \dd r\right)(\xi+1), \qquad \xi \in (-1,0).$$ Then $u' \in L^2(-1,0)$ is given by $$u'(\xi) =   w'(0) - \int_{\xi}^0 g(r) \, \dd r, \qquad \xi \in (-1,0).$$ Moreover, $u\in H^2(-1,0)$ and $u'' = g$,  and we have $u(-1) = 0$ and $u'(0) = w'(0)$. Thus $(u,v,w) \in D(A)$ and $(f,g,h)=A(u,v,w) \in \Ran(A).$
\end{proof}

\section{Neumann boundary condition}\label{sec:Neumann}

In this section we consider the same coupled system but with Neumann boundary for the wave part. We show first that the approach taken in the preceding section is not well suited to studying energy decay and in fact leads to an unbounded semigroup. Consider the following system:
\begin{equation}\label{eq5}
    \begin{cases}\begin{aligned}  u_{tt}(\xi,t) &= u_{\xi\xi}(\xi,t), &&& \xi \in (-1,0),\ t>0,\\
     w_t(\xi,t) &= w_{\xi\xi}(\xi,t), &&& \xi \in (0,\infty),\  t>0, \\
     u_t(0,t) &= w(0,t), &u_\xi(0,t) &= w_\xi(0,t), \;& t>0, \\
     u_\xi(-1,t) &= 0, & && t>0, \\
     u(\xi,0)&=u(\xi), & u_t(\xi,0) &= v(\xi),  & \xi \in (-1,0), \\
     w(\xi,0) &= w(\xi), &&& \xi \in (0,\infty),
    \end{aligned}\end{cases}\end{equation}
where the initial data satisfy $u \in H^1(-1,0), v \in L^2(-1,0)$ and $w \in L^2(0,\infty)$.  This system is the same as that considered in \eqref{eq1}, but with the crucial difference that the Dirichlet boundary condition $u(-1,t)=0$  has been replaced by the Neumann boundary condition $u_\xi(-1,t) = 0$ for all $t>0$. In this case, we consider the Hilbert space $X = H^1(-1,0)\times L^2(-1,0)\times L^2(0,\infty)$, endowed with its natural Hilbert space norm, and define the operator $A$ again by $Ax = (v,u'',w'')$ for $x=(u,v,w)$ in the domain
\begin{equation*}D(A) = \big\{(u,v,w) \in X_0 :  u'(-1) = 0,\ u'(0) = w'(0),\ v(0) = w(0)\big\},\end{equation*}
where $X_0=H^2(-1,0)\times H^1(-1,0)\times H^2(0,\infty)$. As in the Dirichlet case, it can be shown that $A$ is closed and densely defined, and furthermore, the Lumer-Phillips theorem can be applied to $A-I$ to show that $A$ generates a $C_0$-semigroup $(T(t))_{t\geq0}$; see the proof of \cite[Theorem 2.1]{BPS1}. As we shall see shortly, however, this semigroup is no longer a contraction semigroup, and indeed it is not even bounded. We shall arrive at this conclusion by studying the kernel and the range of $A$. We omit the proof of the following proposition, which is very similar to that of Proposition~\ref{ranA}.

\begin{proposition}\label{ranAn}
Let $(f,g,h) \in X$. Then $(f,g,h) \in \Ran(A)$ if and only if the following hold:
\begin{enumerate}
    \item[\textup{(a)}] $\xi \mapsto \displaystyle\lim_{a\to \infty}\int_\xi^a h(r)\, \dd r \in L^2(0,\infty)$;
    \item[\textup{(b)}]  $\xi \mapsto \displaystyle\lim_{b\to \infty}\int_\xi^b \left(\displaystyle\lim_{a\to \infty}\int_t^a h(r)\, \dd r\right)\dd t \in L^2(0,\infty)$;
    \item[\textup{(c)}]  $f(0) = \displaystyle\lim_{b\to \infty}\int_0^b \left(\displaystyle\lim_{a\to \infty}\int_t^a h(r)\,\dd r\right)\dd t$;
    \item[\textup{(d)}]  $\displaystyle\int_{-1}^0 g(r)\, \dd r +\displaystyle\lim_{a\to \infty} \int_0^a h(r)\, \dd r= 0$.
\end{enumerate}
\end{proposition}

Unlike in the Dirichlet case, however, $A$ is no longer injective.

\begin{proposition}\label{kerran}
The following hold:
\begin{enumerate}
    \item[\textup{(a)}]  $\Ker(A) = \spann{\{(1,0,0)\}}$;
    \item[\textup{(b)}]  $\overline{\Ran(A)} = X$.
\end{enumerate}
\end{proposition}

\begin{proof}
Part (a) is easily verified, so we focus on proving (b). Let $(f,g,h) \in X$ and $\ep > 0$. Let $h_0 \in L^2(0,\infty)$ be a compactly supported function such that $\|h-h_0\|_{L^2} < \ep$, and let $r_0 \ge 0$, $\theta \in [0,2\pi)$ be such that
$$r_0 e^{i\theta} = \int_{-1}^0 g(r)\, \dd r +\int_0^\infty h_0(r)\,\dd r.$$ We may find a constant $\xi_0\ge0$ such that
$$\int_0^{\xi_0}\frac{\ep}{\xi+1}\, \dd \xi = r_0,$$ and we define $h_1 \in L^2(0,\infty)$  by
$$h_1(\xi) = -\frac{\ep e^{i\theta}}{\xi+1}\chi_{(0,\xi_0)}(\xi), \qquad \xi >0,$$ where $\chi_I$ denotes the indicator function of the interval $I\subset \R$. Note that $h_1=0$ if $r_0=0$ and that $\|h_1\|_{L^2} < \ep$. Now let
$$c= f(0) - \int_0^\infty \left(\int_t^\infty (h_0(r)+ h_1(r) )\, \dd r\right)\dd t,$$  and let $\sigma\in[0,2\pi)$ be such that $c_0=e^{-i\sigma}c\ge0$. If $c=0$, then it is straightforward to check that the vector $y_0\in X$ defined by $y_0=(f,g,h_0+h_1)$ satisfies conditions (a) through (d) of Proposition~\ref{ranAn}, so that $y_0\in \Ran(A)$, and moreover $\|y-y_0\|_X<\sqrt{2}\ep$. Suppose now that $c\ne0$ and,  given $\tau >0$, define $h_{2,\tau} \in L^2(0,\infty)$ by
$$h_{2,\tau}(\xi) = \frac{\ep e^{i\sigma}}{(1+\xi)^2}\chi_{(0,\tau)}(\xi), \qquad \xi >0.$$ Then $\|h_{2,\tau}\|_{L^2}<\ep/\sqrt{3}$.  Moreover,
$$\int_\xi^\infty h_{2,\tau}(r)\, \dd r = \ep e^{i\sigma}\left(\frac{1}{\xi+1} -\frac{1}{\tau+1}\right), \qquad 0<\xi< \tau,$$
while for $\xi\ge\tau$ the left-hand side equals zero. Hence $$\int_0^\infty \left(\int_t^\infty h_{2,\tau}(r)\, \dd r\right)\dd t = \ep e^{i\sigma}\left(\log(\tau+1) - \frac{\tau}{\tau+1}\right).$$ We now make the choice $\tau = e^{c_0/\ep}-1$ and define $y_0\in X$ by
$${y_0} = \left(f-\frac{\ep\tau e^{i\sigma}}{\tau+1},g-\frac{\ep\tau e^{i\sigma}}{\tau+1},h_0+h_1  + h_{2,\tau}\right).$$
It follows  from Proposition~\ref{ranAn} that $y_0\in\Ran(A)$, and furthermore $\|y-y_0\|_X<3\ep$. Since $\ep>0$ was arbitrary, $\Ran(A)$ is dense in $X$.
\end{proof}

It is a well known fact in ergodic theory that if  $(T(t))_{t\geq0}$ is a bounded on a Hilbert space $X$ (or more generally on a reflexive Banach space), then \begin{equation}\label{met}X=\overline{\Ran(A)}\oplus \Ker(A);\end{equation} see for instance \cite[Section~4.3]{ABHN}. By Proposition~\ref{kerran} this splitting fails to hold for our operator $A$, which allows us to make the following observation.

\begin{corollary}
The semigroup $(T(t))_{t\geq0}$ generated by $A$ is unbounded.
\end{corollary}

\begin{remark}
We do not require the full strength of the splitting \eqref{met} in order to show that $A$ generates an unbounded semigroup, and in fact it is enough to show that ${\overline{\Ran(A)}}\cap \Ker(A) \neq \{0\}$; see \cite[Section~4.3]{ABHN} again. This can be done somewhat more directly by proving that $(1,0,0) \in {\overline{\Ran(A)}}$. Note also that since the semigroup in the Dirichlet case is bounded, \eqref{met} implies that  the generator has dense range in that case.
\end{remark}

 The underlying reason why the semigroup $(T(t))_{t\geq0}$ considered here is unbounded is that the wave equation on a finite interval with Neumann boundary conditions at both ends by itself gives rise to an unbounded semigroup, and the damping provided by the coupling with the  heat equation on the half-line is insufficient to make $(T(t))_{t\geq0}$ bounded. By contrast, if we replace the infinite heat part with a heat equation on a \emph{finite} interval, which is precisely the case considered in \cite{BPS1}, then the damping is sufficiently strong to make the semigroup corresponding to the coupled system bounded; see also Remark~\ref{rem:opt}(b). 
 
 Observe, however, that the norm on $X$ contains the $L^2$-norm of the displacement of the wave, which does not appear in the physically motivated definition of the energy given in \eqref{eq:energy}. This leads us to consider an alternative formulation of our wave-heat system. Indeed, let $Y=L^2(-1,0)\times L^2(-1,0)\times L^2(0,\infty)$, endowed with its natural Hilbert space norm, and define the operator $B$ on $Y$ by $By=(v',u', w'')$ for $y=(u,v,w)$ in the domain 
 $$D(B)=\big\{(u,v,w)\in Y_0:u(-1)=0,\ u(0)=w'(0),\ v(0)=w(0)\big\},$$
 where $Y_0=H^1(-1,0)\times H^1(-1,0)\times H^2(0,\infty)$. Following the same steps as in Section~\ref{sec:gen} we may show that $B$ is the generator of a contractive $C_0$-semigroup $(S(t))_{t\geq0}$ on $Y$. Given $y=(u,v,w)\in Y$ the orbit $\{ S(t)y:t\ge0\}$ is the solution of the wave-heat system in which $w$ is the initial heat profile, $v$ is the initial wave velocity, as before, but $u$ is now the initial \emph{slope} of the wave part rather than its \emph{displacement}. It is due to the fact that the displacement of the wave part no longer features in our formulation that, unlike in the previous setup, we now obtain a bounded semigroup. Note also that with only a slight abuse of notation we may write the energy of the solution corresponding to the initial data $y\in Y$ as 
 $$E_y(t)=\frac12\|S(t)y\|_{Y}^2,\quad t\ge0.$$
 We may now proceed as in the case of Dirichlet boundary condition to study the spectrum of the generator $B$, the growth of the resolvent operator along the imaginary axis and the resulting decay rates for suitable orbits of the semigroup $(S(t))_{t\ge0}$. The arguments are entirely analogous to those presented in Sections~\ref{sec:gen}, \ref{sec:res} and \ref{sec:en}, requiring only minor modifications but no new ideas. We summarise the relevant statements as follows.
 
 \begin{theorem}\label{thm:Neumann}
 The spectrum of the generator $B$ is the disjoint union 
$$\sigma(B) = (-\infty,0]\cup\sigma_p(B),$$ 
where 
\begin{align*}\sigma_p(B) = \big\{\lb \in \C_- :\slb \cosh(\lb)+\sinh(\lb)  = 0\big\}.\end{align*}
In particular, the spectrum satisfies $\sigma(B) \cap i\R = \{0\}$ and we have
\begin{equation}\label{eq:res_cases}
\|R(is,B)\|=\begin{cases}
O(|s|^{1/2}), & |s|\to\infty,\\
O(|s|^{-1}), &|s|\to0.
\end{cases}\end{equation}
For all $y\in Y$ we have $E_y(t)\to0$ as $t\to\infty$, and if $y\in D(B)\cap\Ran(B)$ then $E_y(t)=O(t^{-2})$ as $t\to\infty$. Finally, $\Ran(B)$ consists of all those $(f,g,h)\in Y$ satisfying 
\begin{enumerate}
    \item[\textup{(a)}] $\xi \mapsto \displaystyle\lim_{a\to \infty}\int_\xi^a h(r)\, \dd r \in L^2(0,\infty)$;
    \item[\textup{(b)}]  $\xi \mapsto \displaystyle\lim_{b\to \infty}\int_\xi^b \left(\displaystyle\lim_{a\to \infty}\int_t^a h(r)\, \dd r\right)\dd t \in L^2(0,\infty)$;
    \item[\textup{(c)}]  $\displaystyle\int_{-1}^0 g(r)\,\dd r +\displaystyle\displaystyle\lim_{a\to \infty}\int_0^a h(r)\, \dd r=0$.
\end{enumerate} 
 \end{theorem}

\begin{remark}
As in the Dirichlet case, the energy decay rate in Theorem~\ref{thm:Neumann} is optimal in the same sense as in Theorem~\ref{opt} and the resolvent estimates in \eqref{eq:res_cases} are sharp; see Remark~\ref{rouche}.
\end{remark}

\end{document}